\documentclass[12pt,reqno]{amsart}

\usepackage{amsmath, amsthm, amssymb,amsfonts, commath, mathrsfs, csquotes} \usepackage{xcolor}

\DeclareMathOperator{\diam}{diam}
\newcommand{\N}{\mathbb{N}}

\newcommand{\Q}{\mathbb{Q}}
\newcommand{\R}{\mathbb{R}}
\renewcommand{\subset}{\subseteq}
\renewcommand{\supset}{\supseteq}

\usepackage[nameinlink,capitalize]{cleveref}
\newtheorem{thm}{Theorem} [section]
\newtheorem{lem}[thm]{Lemma}
\newtheorem{prop}[thm]{Proposition}
\newtheorem{cor}[thm]{Corollary}
\newtheorem{conj}[thm]{Conjecture}
\newtheorem{claim}[thm]{Claim}
\theoremstyle{definition} 

\newtheorem{exmp}[thm]{Example}
\theoremstyle{remark}
\newtheorem*{remark}{Remark}
\newtheorem*{remarks}{Remarks}

\numberwithin{equation}{section}

\begin{document}
	
\title[Zero-full law in missing digit sets]{Zero-full law for well approximable sets in missing digit sets}

\author{Bing Li} 
\email{scbingli@scut.edu.cn}
\address{School of Mathematics, South China University of Technology, Guangzhou, China, 510641}
\author{Ruofan Li} 
\email{liruofan@jnu.edu.cn}
\address{Department of Mathematics, Jinan University, Guangzhou, China, 510632}
\author{Yufeng Wu}
\email{yufengwu.wu@gmail.com}
\address{School of Mathematics and Statistics, Central South University, Changsha, China, 410083}

\begin{abstract}
Let $b \geq 3$ be an integer and $C(b,D)$ be the set of real numbers in $[0,1]$ whose base $b$ expansion only consists of digits in a set $D \subset \{0,...,b-1\}$. We study how close can numbers in $C(b,D)$ be approximated by rational numbers with denominators being powers of some integer $t$ and obtain a zero-full law for its Hausdorff measure in several circumstances. When $b$ and $t$ are multiplicatively dependent, our results correct an error of Levesley, Salp and Velani ({\em Math. Ann.}, 338:97–118, 2007) and generalize their theorem. When $b$ and $t$ are multiplicatively independent but have the same prime divisors, we obtain a partial result on the Hausdorff measure and bounds for the Hausdorff dimension, which are close to the multiplicatively dependent case. Based on these results, several conjectures are proposed.
\end{abstract}
\maketitle

\section{Introduction}
In an influential article \cite{Mahler84}, Mahler asked \enquote{How close can irrational elements of Cantor's set be approximated by rational numbers?}. This question has inspired a wide range of research, such as \cite{Bugeaud08,FS14,LLW,Schleischitz21} and references therein. In \cite{LSV07}, Levesley, Salp and Velani showed that there exist numbers which are not Liouville numbers in the middle-third Cantor set $C$ that can be very well approximated by rational numbers whose denominators are powers of $3$. More precisely, let $\psi : \N \to (0,\infty)$ be a function and define
\begin{align*}
W_{3} (\psi) = \left\{x \in [0,1] \colon \abs{x - \frac{p}{3^n}} < \psi(n) \text{ for i.m. } (p,n) \in \N^2\right\},
\end{align*}
where i.m. is short for \enquote{infinitely many}. They proved that the $f$-Hausdorff measure $\mathcal{H}^{f}$ (we refer to \cref{sec:Pre} for terminologies) of $W_{3} (\psi) \cap C$ satisfies a zero-full law.

\begin{thm}[{\cite[Theorem 1]{LSV07}}] \label{thm:LSV}
Let $f$ be a dimension function such that $f(r) r^{-\log2 / \log 3} $ is monotonic and $\psi : \N \to (0,\infty)$ be a function. Then
\begin{align*}
\mathcal{H}^{f} (W_{3}(\psi) \cap C) = \begin{cases}
0, &\text{ if } \sum_{n=1}^{\infty} f(\psi(n)) \times 3^{n \log2 / \log 3} < \infty, \\
\mathcal{H}^{f}(C), &\text{ if } \sum_{n=1}^{\infty} f(\psi(n)) \times 3^{n \log2 / \log 3} = \infty.
\end{cases}
\end{align*}
\end{thm}

In the same article, the authors claimed their method can also yield a generalization of \cref{thm:LSV}. For any integer $b \geq 3$ and set $D \subset \{0,...,b-1\}$ with cardinality between $2$ and $b-1$, the \textit{missing digit set} $C(b,D)$ is defined as the set of real numbers in $[0,1]$ whose base $b$ expansions only consist of digits in $D$. 

\begin{claim}[{\cite[Theorem 4]{LSV07}}] \label{claim}
Let $C(b,D)$ be a missing digit set with Hausdorff dimension $\gamma$, $f$ be a dimension function such that $r^{-\gamma}f(r)$ is monotonic and $\psi : \N \to (0,\infty)$ be a function. Then
\begin{align*}
\mathcal{H}^{f} (W_{b}(\psi) \cap C(b,D)) = \begin{cases}
0, \hspace{-5.4pt}&\text{ if } \sum_{n=1}^{\infty} f(\psi(n)) b^{n \gamma} < \infty, \\
\mathcal{H}^{f}(C(b,D)), \hspace{-5.4pt}&\text{ if } \sum_{n=1}^{\infty} f(\psi(n)) b^{n \gamma} = \infty.
\end{cases}
\end{align*}
\end{claim}

Unfortunately, this claim is not always valid. When the set $D$ does not contain $0$ and $b-1$, all rational numbers of the form $p b^{-n}$ are not in the missing digit set $C(b,D)$, so it is possible to find $\psi$ such that $\sum_{n=1}^{\infty} f(\psi(n)) \times b^{n \gamma}$ diverges while $W_{b}(\psi) \cap C(b,D) = \emptyset$; see Example \ref{exmp} for a concrete example. The correction of Claim \ref{claim} is as follows.

\begin{thm} \label{thm:b}
Let $C(b,D)$ be a missing digit set with Hausdorff dimension $\gamma$, $m = \min\{ \min D, b-1-\max D \},$ $f$ be a dimension function such that $r^{-\gamma} f(r)$ is monotonic, and $\psi : \N \to (0,\infty)$ be a function. Then $\mathcal{H}^{f} (W_{b}(\psi) \cap C(b,D)) =$
\begin{align*}
\begin{cases}
0, &\text{ if } \sum\limits_{n \geq 1 \colon \psi(n) > \frac{m}{(b-1)b^n}} f\left(\psi(n) - \frac{m}{(b-1)b^n}\right)  b^{n \gamma} < \infty, \\
\mathcal{H}^{f}(C(b,D)), &\text{ if } \sum\limits_{n \geq 1 \colon \psi(n) > \frac{m}{(b-1)b^n}}  f\left(\psi(n) - \frac{m}{(b-1)b^n}\right) b^{n \gamma} = \infty.
\end{cases}
\end{align*}
\end{thm}

It is crucial in \cref{thm:b} to approximate numbers in $C(b,D)$ by rational numbers with denominators $b^n$, as this matches the structure of $C(b,D)$. One naturally wonders if similar zero-full laws hold when denominators are powers of other numbers. Let $t \geq 2$ be an integer, we consider the Hausdorff measure of $W_{t}(\psi) \cap C(b,D)$, where 
$$W_{t} (\psi) = \left\{x \in [0,1] \colon \abs{x - \frac{p}{t^n}} < \psi(n) \text{ for i.m. } (p,n) \in \N^2\right\}.$$
For $t=2$, Velani proposed the following conjecture.

\begin{conj}[Velani's conjecture; see \cite{ACY20}] \label{conj:Velani}
Suppose $\psi: \N \to (0,\infty)$ is monotonic, then
\begin{align*}
\mathcal{H}^{\log2 / \log 3} (W_{2}(\psi) \cap C) = \begin{cases}
0, &\text{ if } \sum_{n=1}^{\infty} \psi(n) \times 2^n < \infty, \\
1, &\text{ if } \sum_{n=1}^{\infty} \psi(n) \times 2^n = \infty.
\end{cases}
\end{align*}
\end{conj}

Some progresses towards proving Conjecture \ref{conj:Velani} have been made by various authors; see \cite{ABCY22,ACY20,Baker22}. The main difficultly is that we know very little on how the dyadic rationals are distributed around the middle-third Cantor set, this is similar to the Furstenberg's conjecture on times two and times three \cite{Furstenberg70}.

We are going to consider the case that $b$ and $t$ are multiplicatively dependent, that is, $\log t / \log b \in \Q$. Under this assumption, we prove a zero-full law.

\begin{thm} \label{thm:dependent}
Let $C(b,D)$ be a missing digit set with Hausdorff dimension $\gamma$ such that $D$ contains at least one of $0$ and $b-1$, $t \geq 2$ be an integer which is multiplicatively dependent with $b$, $f$ be a dimension function such that $r^{-\gamma} f(r)$ is monotonic, and $\psi : \N \to (0,\infty)$ be a function. Then
\begin{align*}
\mathcal{H}^{f} (W_{t}(\psi) \cap C(b,D)) = \begin{cases}
0, &\text{ if } \sum_{n=1}^{\infty} f(\psi(n))  t^{n \gamma} < \infty, \\
\mathcal{H}^{f}(C(b,D)), &\text{ if } \sum_{n=1}^{\infty}  f(\psi(n)) t^{n \gamma} = \infty.
\end{cases}
\end{align*}
\end{thm}

Unlike \cref{thm:b}, here we need the assumption that $D$ contains at least one of $0$ and $b-1$ to obtain a complete zero-full law. If this condition is dropped, we are still able to deduce a result for $\mathcal{H}^{f} (W_{t}(\psi) \cap C(b,D))$, but the two series for the divergence and convergence parts may be different. Indeed, our method is applicable to more general approximable sets in the case that $b$ and $t$ have the same prime divisors, which is weaker than that $p$ and $t$ are multiplicatively dependent.

Suppose $A=(a_{n})_{n \geq 1}$ is a sequence of positive integers and define
\begin{align*}
W_{t,A}(\psi) =  \left\{x \in [0,1] \colon \abs{x - \frac{p}{t^{a_{n}}}} < \psi(n) \text{ for i.m. } (p,n) \in \N^2\right\}.
\end{align*}
Let 
\begin{align*}
I(A) = \{i \in \N \colon a_{n} = i \text{ for some } n\}.
\end{align*}
That is, $I(A)$ is the set of elements of the sequence $A$. For $i\in I(A)$, let
\begin{align*}
\psi_{A}(i) = \max \left\{\psi(n) \colon a_{n}=i \right\}.
\end{align*}
If the prime divisors of $b$ and $t$ are the same, denote 
\begin{align*}
\alpha_{1}(b,t) = \min\left\{\frac{v_{q}(t)}{v_{q}(b)} \colon q \text{ is a prime divisor of } b \right\},\\
\alpha_{2}(b,t) = \max\left\{\frac{v_{q}(t)}{v_{q}(b)} \colon q \text{ is a prime divisor of } b  \right\},
\end{align*}
where $v_{q}(b)$ means the greatest integer such that $q^{v_{q}(b)}$ divides $b$. For the sake of clarity, we will simply write $I(A)$, $\alpha_{1}(b,t)$ and $\alpha_{2}(b,t)$ as $I$, $\alpha_{1}$ and $\alpha_{2}$ respectively when there is no confusion. We use the notations $\lfloor \cdot \rfloor$ and $\lceil \cdot \rceil$ to mean the floor and ceiling functions respectively.

\begin{thm} \label{thm:main}
Suppose $C(b,D)$ is a missing digit set with Hausdorff dimension $\gamma$, $m$ is the greatest integer such that $$D \subset \{m, m+1, \ldots, b-1-m\},$$ $t \geq 2$ is an integer that has the same prime divisors as $b$, and $A=(a_{n})_{n \geq 1}$ is an unbounded non-decreasing sequence of positive integers. Let $f$ be a dimension function such that $r^{-\gamma} f(r)$ is monotonic, and $\psi : \N \to (0,\infty)$ be a function. Then $\mathcal{H}^{f} (W_{t,A}(\psi) \cap C(b,D)) = $
\begin{align*}
\begin{cases}
0 \hspace{-5.0pt }&\text{ if } \sum\limits_{i \in I \colon \psi_{A}(i) > \frac{m}{(b-1) b^{\lceil i \alpha_{2} \rceil }}} f\left(\psi_{A}(i) - \frac{m}{(b-1) b^{\lceil i \alpha_{2} \rceil }}\right) b^{i \alpha_{2} \gamma } < \infty, \\
\mathcal{H}^{f}(C(b,D)) \hspace{-5.0pt }&\text{ if } \sum\limits_{i \in I \colon \psi_{A}(i) > \frac{m}{(b-1) b^{\lfloor i \alpha_{1} \rfloor }}} f\left(\psi_{A}(i) - \frac{m}{(b-1) b^{\lfloor i \alpha_{1} \rfloor }}\right) b^{i \alpha_{1} \gamma} = \infty.
\end{cases}
\end{align*}
\end{thm}

\begin{remarks}
{\rm (i)} When $b$ and $t$ are multiplicatively dependent, we have $$\alpha_{1} = \alpha_{2} = \frac{\log t}{\log b}$$ and thus \cref{thm:dependent} follows from \cref{thm:main} by taking $A = (n)_{n \geq 1}$ and using $m=0$. 

{\rm (ii)} If $b=t$ and $A = (n)_{n \geq 1}$, then $\alpha_{1} = \alpha_{2} = 1$ and hence $\lceil n \alpha_{2} \rceil = \lfloor n \alpha_{1} \rfloor = n$. Therefore \cref{thm:main} also implies \cref{thm:b}.

{\rm (iii)} In general, the two series in \cref{thm:main} are different, so our formula is inconclusive in the case that the first series diverges while the second series converges.
\end{remarks}

The rest of this article is structured as follows. \cref{sec:Pre} includes terminologies and tools needed. The main reason that Claim \ref{claim} is wrong and what modification is needed are discussed in \cref{sec:inter}. We prove a variation of \cref{thm:main} in \cref{sec:A} and then use it to obtain \cref{thm:main} in \cref{sec:dependent}. Finally, we discuss Hausdorff dimension and the large intersection property of $W_{t}(\psi) \cap C(b,D)$ and propose a conjecture for multiplicatively independent case in \cref{sec:further}.

\section{Preliminaries} \label{sec:Pre}
\subsection{Hausdorff measure and dimension}
A function $f:(0,\infty) \to (0,\infty)$ is called a \textit{dimension function} if it is continuous, non-decreasing and $\lim_{r \to 0} f(r) = 0$. For a set $S \subset \R^k$, we say a countable collection of balls $\{B_{i}\}$ in $\R^{k}$ is a \textit{$\rho$-cover} of $S$ if $S \subset \cup_{i} B_{i}$ and their radii are not larger than $\rho$. The \textit{Hausdorff $f$-measure} $\mathcal{H}^{f}$ of $S$ is 
$$\mathcal{H}^{f}(S) = \lim_{\rho \to 0} \mathcal{H}^{f}_{\rho} (S),$$
where
$$\mathcal{H}^{f}_{\rho} (S) = \inf\left\{ \sum_{i} f(r(B_{i})) \colon \{B_{i}\} \text{ is a $\rho$-cover of } S \right\},$$
and $r(B_{i})$ means the radius of ball $B_{i}$.

When $f(r) = r^{s}$ for some $s \geq 0$, we write $\mathcal{H}^{f}$ as $\mathcal{H}^{s}$. The \textit{Hausdorff dimension} of a set $S$ is
$$\dim_{\rm H} S = \inf\{s \colon \mathcal{H}^{s}(S) = 0\}.$$
It is known that a missing digit set $C(b,D)$ has Hausdorff dimension $\log \# D / \log b$, where $\# D$ denotes the cardinality of $D$. This number will be used frequently and we denote it by $\gamma$. More properties of Hausdorff measure and dimension can be found in \cite{Falconer90}.

\subsection{Mass transference principle}
For two positive numbers $x$ and $y$, we write $x \ll y$ if there exists a constant $K>0$ such that $x \leq K y$. The relation $x \gg y$ is defined similarly and we write $x \asymp y$ if $x \ll y$ and $x \gg y$.

Let $X$ be a compact subset of $\R^{k}$ and $\mu$ be a Borel measure on $X$. We say $\mu$ is \textit{$\delta$-Ahlfors regular} if there exists constant $r_{0}>0$ such that for any ball $B(x,r) \subset X$ with $x \in X$ and radius $r \leq r_{0}$, we have $$\mu(B(x,r)) \asymp r^{\delta}.$$  When $X$ is a missing digit set $C(b,D)$ with dimension $\gamma$, the measure $\mathcal{H}^{\gamma}\vert_{C(b,D)}$ is $\gamma$-Ahlfors regular; see for example \cite{Mattila95}. This allows us to use the mass transference principle, a widely-used tool in computing Hausdorff dimension of limsup sets.

\begin{thm}[Mass transference principle, \cite{BV06}] \label{thm:MTP}
Let $X$ be a compact subset of $\R^{k}$ equipped with a $\delta$-Ahlfors regular measure $\mu$. Let $(B_{n})_{n \geq 1}$ be a sequence of balls in $X$ with $r(B_n) \to 0$ as $n \to \infty$. Suppose $f$ is a dimension function such that $r^{-\delta} f(r)$ is monotonic. For a ball $B=B(x,r)$, denote $B^{f} = B(x, f(r)^{1/\delta})$. 
If for any ball $B$ in $X$, we have 
$$ \mathcal{H}^{\delta} (B \cap \limsup_{n \to \infty} B_{n}^{f}) = \mathcal{H}^{\delta} (B),$$
then
$$ \mathcal{H}^{f} (B \cap \limsup_{n \to \infty} B_{n}) = \mathcal{H}^{f} (B)$$
for any ball $B$ in $X$.
\end{thm}

\subsection{Measure theoretic lemmas}
In this subsection we state several lemmas on measures. The first one is about when a subset has the same measure as the whole set.

\begin{lem}[{\cite[Lemma 1]{LSV07}}] \label{lem:full_measure}
Let $X$ be a compact set in $\R^{k}$ and $\mu$ be a finite measure on $X$ such that all open sets are measurable and $\mu(B(x,2r)) \ll \mu(B(x,r))$ for all balls $B(x,r)$ with center in $X$. Suppose $E$ is a Borel subset of $X$ and there exist positive constants $r_{0}$, $c_{0}$ such that for any ball $B$ with radius $r(B) < r_{0}$ and center in $X$, we have $\mu(E \cap B) \geq c_{0} \mu(B)$. Then $$\mu(E) = \mu(X).$$
\end{lem}

The second lemma is a generalization of the divergence part of the Borel-Cantelli lemma.

\begin{lem} [{\cite[Lemma 5]{Sprindžuk79}}; see also {\cite[Lemma 2]{LSV07}}] \label{lem:Div_BC}
Let $X$ be a compact set in $\R^k$ and let $\mu$ be a finite measure on $X$. Also, let $E_n$ be a sequence of $\mu$-measurable sets such that
$\sum_{n=1}^{\infty} \mu(E_n) = \infty$. Then
$$\mu (\limsup_{n \to \infty} E_n) \geq \limsup_{Q \to \infty} \frac{\left(\sum_{0<s \leq Q} \mu(E_s)\right)^2}{\sum_{0<s,t \leq Q} \mu(E_s \cap E_t)}.$$
\end{lem}

In \cref{thm:main}, an infinite countable set $I$ is used, and we have the following variation of Lemma \ref{lem:Div_BC}: If $\sum_{i \in I} \mu(E_i) = \infty$, then
\begin{align}
\mu (\limsup_{i \to \infty, i \in I} E_i) \geq \limsup_{Q \to \infty} \frac{\left(\sum_{0<s \leq Q, s \in I} \mu(E_s)\right)^2}{\sum_{0<s,t \leq Q, s,t \in I} \mu(E_s \cap E_t)}. \label{eq:Div_BC}
\end{align}

\section{Intersection of balls and the missing digit set} \label{sec:inter}
The set $W_{b}(\psi)$ is the limsup of balls of the form $B(p b^{-n}, \psi(n))$, so we investigate what is the intersection of those balls and $C(b,D)$. We start with an example illustrating why Claim \ref{claim} is false.

\begin{exmp} \label{exmp}
Let $b = 5$, $D = \{1,2\}$, $\gamma = \log 2 / \log 5$ be the dimension of the missing digit set $C(5, \{1,2\})$ and
$$\psi(n) = \sum_{i=n+1}^{\infty} \frac{1}{5^i} = \frac{1}{4 \times 5^n}.$$
Since
$$\sum_{n=1}^{\infty} \left(\frac{1}{4 \times 5^n}\right)^{\gamma} \times 5^{n \gamma} = \sum_{n=1}^{\infty} \frac{1}{4^{\gamma}} = \infty,$$
Claim \ref{claim} says $$\mathcal{H}^{\gamma} \left(W_{5}\left(\psi\right) \cap C(5, \{1,2\})\right) =\mathcal{H}^{\gamma}(C(5, \{1,2\})) >0.$$ 

Let $x \in [0,1]$ be a number in $B(p 5^{-n}, \psi(n))$ for some $(p,n) \in \N^2$, then
$$\frac{p-1}{5^n} + \sum_{i=n+1}^{\infty} \frac{3}{5^i} < x < \frac{p}{5^n} + \sum_{i=n+1}^{\infty} \frac{1}{5^i},$$
hence the base $5$ expansion of $x$ must contain a digit $0$ or $4$, and thus $x \notin C(5, \{1,2\})$. Therefore in this case 
$$W_{5}\left(\psi\right) \cap C(5, \{1,2\}) = \emptyset$$ cannot have positive measure.
\end{exmp}

When the set $D$ does not contain $0$ and $b-1$, we have $p b^{-n} \notin C(b,D)$ for any $(p,n) \in \N^2$, so the intersections of balls $B(p b^{-n}, \psi(n))$ and $C(b,D)$ are all empty unless $\psi(n)$ is not too small; see the proof of Lemma \ref{lem:intersection} for details. To better understand how those intersections are, we introduce several notations. For any $n \geq 1$, let $C_{n}(b,D)$ be the $n$-th level of $C(b,D)$, which consists of $b^{n \gamma}$ intervals of length $b^{-n}$. More precisely, $$C_{n}(b,D) = \left\{\sum_{i=1}^{\infty} \frac{x_{i}}{b^{i}} \in [0,1] \colon x_{i} \in D \text{ for } i=1,...,n\right\}.$$

Denote the set of all left endpoints of the intervals in $C_{n}(b,D)$ by $L_{n}$ and the set of all right endpoints of the intervals in $C_{n}(b,D)$ by $R_{n}$. Note that a point can be both a left and right endpoint. For instance, in Example \ref{exmp}, we have $L_{1} = \{1/5, 2/5\}$ and $R_{1} = \{2/5, 3/5\}$, hence $2/5$ is in both $L_{1}$ and $R_{1}$. The two quantities below are used to measure how many digits of $\{0,\ldots,b-1\}$ are missing in $D$ from left and right respectively.
\begin{align*}
m_{l} = \min\{D\} \quad \text{ and } \quad m_{r} = b-1-\max\{D\}.
\end{align*}
Recall that $m = \min\{m_{l}, m_{r}\}$.

\begin{lem} \label{lem:intersection}
Suppose $\psi(n) < b^{-n}/2$. Denote $$d_{l,n} = \frac{m_{l}}{(b-1)b^n} \quad \text{ and } \quad d_{r,n} = \frac{m_{r}}{(b-1)b^n}.$$ We make the convention that an open ball with non-positive radius is regarded as an empty set.
\begin{enumerate}
\item If $p b^{-n} \in L_{n} \setminus R_{n}$, then
$$B\left(\frac{p}{b^{n}},\psi(n)\right) \cap C(b,D) = B\left(\frac{p}{b^{n}} + d_{l,n},\psi(n)-d_{l,n}\right) \cap C(b,D).$$

\item If $p b^{-n} \in R_{n} \setminus L_{n}$, then
$$B\left(\frac{p}{b^{n}},\psi(n)\right) \cap C(b,D) = B\left(\frac{p}{b^{n}} - d_{r,n},\psi(n)-d_{r,n}\right) \cap C(b,D).$$

\item If $p b^{-n} \in L_{n} \cup R_{n}$, then
\begin{align*}
&B\left(\frac{p}{b^{n}},\psi(n)\right) \cap C(b,D) \\
= &\left(B\left(\frac{p}{b^{n}} + d_{l,n},\psi(n)-d_{l,n}\right) \cup B\left(\frac{p}{b^{n}} - d_{r,n},\psi(n)-d_{r,n}\right) \right) \\
&\cap C(b,D).
\end{align*}

\item If $p b^{-n} \notin L_{n} \cup R_{n}$, then $$B\left(\frac{p}{b^{n}},\psi(n)\right) \cap C(b,D) = \emptyset.$$
\end{enumerate}
\end{lem}
\begin{proof}
\begin{enumerate}
\item If $m_{l}=0$, then $d_{l,n}=0$ and the equality holds trivially. So we assume $m_{l}>0$ without loss of generality. Let $x \in B(p b^{-n}, \psi(n))$. If $x<p b^{-n}$, then $\psi(n) < b^{-n}/2$ implies that $x \notin C_{n}(b,D) \supset C(b,D).$ If $p b^{-n} \leq x < p b^{-n} + d_{l,n}$, then $$\frac{p}{b^n} \leq x < \frac{p}{b^n} + \frac{m_{l}}{(b-1)b^n} =  \frac{p}{b^n} + \sum_{i=n+1}^{\infty} \frac{m_{l}}{b^{i}},$$
so the base $b$ expansion of $x$ contains a digit between $0$ and $m_{l}-1$, hence $x \notin C(b,D)$ by the definition of $m_{l}$. Then $x \in C(b,D)$ only happens when $x \geq p b^{-n} + d_{l,n}$, which implies that $\psi(n) > d_{l,n}$. Also note that if $x \geq p b^{-n} + d_{l,n}$, then $$0 \leq x - \left(\frac{p}{b^{n}} + d_{l,n}\right) = x - \frac{p}{b^{n}} - d_{l,n} < \psi(n)-d_{l,n},$$ so $x \in B\left(p b^{-n} + d_{l,n},\psi(n)-d_{l,n}\right)$. Therefore if $\psi(n) > d_{l,n}$, then $$B\left(\frac{p}{b^{n}},\psi(n)\right) \cap C(b,D) \subset B\left(\frac{p}{b^{n}} + d_{l,n},\psi(n)-d_{l,n}\right) \cap C(b,D)$$ and these two sets are equal since $$B\left(\frac{p}{b^{n}} + d_{l,n},\psi(n)-d_{l,n}\right) \subset B\left(\frac{p}{b^{n}},\psi(n)\right)$$ is trivial.

\item Again, we may assume $m_{r}>0$ without loss of generality. Let $x \in B(p b^{-n}, \psi(n))$. If $x>p b^{-n}$, then $x \notin C_{n}(b,D)$ since $\psi(n) < b^{-n}/2$. If $p b^{-n}-d_{r,n} \leq x < p b^{-n}$, then $$\frac{p-1}{b^n} + \sum_{i=n+1}^{\infty} \frac{b-1-m_{r}}{b^{i}} \leq x < \frac{p}{b^n},$$ so the base $b$ expansion of $x$ contains a digit between $b-m_{l}$ and $b-1$, hence $x \notin C(b,D)$ by the definition of $m_{r}$. Then a similar argument as in the previous case shows that $$B\left(\frac{p}{b^{n}},\psi(n)\right) \cap C(b,D) = B\left(\frac{p}{b^{n}} - d_{r,n},\psi(n)-d_{r,n}\right) \cap C(b,D)$$ if $\psi(n) > d_{r,n}$ and the intersection on the left is empty when $\psi(n) \leq d_{r,n}$.

\item This part is treated by combining the arguments in previous two cases, hence we skip the details. 

\item If $p b^{-n} \notin R_{n} \cup L_{n}$, then $\psi(n) < b^{-n}/2$ implies that $$B\left(p b^{-n},\psi(n)\right) \cap C_{n}(b,D) = \emptyset,$$ and thus $B\left(p b^{-n},\psi(n)\right) \cap C(b,D) = \emptyset.$
\end{enumerate}
\end{proof}

By the definition of $m_{l}$ and $m_{r}$, the new ball centers $p b^{-n} + d_{l,n}$ and $p b^{-n} - d_{l,n}$ are points in the missing digit set $C(b,D)$, hence if $\mu = \mathcal{H}^{\gamma} \mid_{C(b,D)}$, then 
\begin{align}
\mu\left(B\left(\frac{p}{b^{n}} + d_{l,n},\psi(n)-d_{l,n}\right)\right) \asymp (\psi(n)-d_{l,n})^{\gamma}, \label{eq:mu_left} \\
\mu\left(B\left(\frac{p}{b^{n}} - d_{r,n},\psi(n)-d_{r,n}\right)\right) \asymp (\psi(n)-d_{r,n})^{\gamma}. \label{eq:mu_right}
\end{align}

\section{A special case} \label{sec:A}
In this section we prove a special case of \cref{thm:main} where $b=t$. It will be used later to deduce \cref{thm:main}.
\begin{lem} \label{thm:A}
Suppose $C(b,D)$ is a missing digit set with Hausdorff dimension $\gamma$, $m = \min\{ \min D, b-1-\max D \},$ and $A=(a_{n})_{n \geq 1}$ is an unbounded non-decreasing sequence of positive integers. Let $f$ be a dimension function such that $r^{-\gamma} f(r)$ is monotonic, and $\psi : \N \to (0,\infty)$ be a function. Then $\mathcal{H}^{f} (W_{b,A}(\psi) \cap C(b,D)) =$
\begin{align*}
\begin{cases}
0, &\text{ if }  \sum\limits_{i \in I \colon \psi_{A}(i) > \frac{m}{(b-1) b^{ i }}} f\left(\psi_{A}(i) - \frac{m}{(b-1) b^{i}}\right) b^{i \gamma } < \infty, \\
\mathcal{H}^{f}(C(b,D)), &\text{ if }  \sum\limits_{i \in I \colon \psi_{A}(i) > \frac{m}{(b-1) b^{i }}} f\left(\psi_{A}(i) - \frac{m}{(b-1) b^{i}}\right) b^{i \gamma} = \infty.
\end{cases}
\end{align*}
\end{lem}

\begin{remark}
We can assume $\psi_{A}(i) < b^{-i}/2$ for all $i$ without loss of generality. Indeed, suppose there exists an infinite set $I_{0}$ such that $\psi_{A}(i) \geq b^{-i}/2$ for all $i \in I_{0}$, thus $W_{b,A}(\psi) = [0,1]$. Since the cardinality of $D$ is at least $2$, we have $m < (b-1)/2$. Then
\begin{align}
f\left(\psi_{A}(i) - \frac{m}{(b-1) b^{i}}\right) b^{i \gamma} &\geq f\left(\frac{1}{2b^{i}} - \frac{m}{(b-1) b^{i}}\right) b^{i \gamma} \nonumber\\
&= f\left(\frac{b-1-2m}{2(b-1) b^{i}}\right) b^{i \gamma} \nonumber\\
&\geq f\left(\frac{1}{2(b-1) b^{i}}\right) b^{i \gamma} \nonumber \\
&\geq b^{-\gamma} f\left(\frac{1}{2 b^{i+1}}\right) b^{(i+1) \gamma}. \label{eq:remark}
\end{align}
for all $i \in I_{0}$. For any $\rho > 0$ and any integer $i_{0}$ big enough, we have
\begin{align*}
\mathcal{H}^{f}_{\rho}(C(b,D)) \ll& \sum_{i \geq i_{0} \colon i \in I_{0}} f\left(\frac{1}{2b^{i}}\right) b^{i \gamma}.
\end{align*}
If $\mathcal{H}^{f}(C(b,D)) = 0$, then \cref{thm:A} is trivial. Otherwise, $$\sum\limits_{i \in I_{0}} f\left(\frac{1}{2 b^{i}}\right) b^{i \gamma} = \infty,$$ which implies $$\sum\limits_{i \in I \colon \psi_{A}(i) > \frac{m}{(b-1) b^{i }}} f\left(\psi_{A}(i) - \frac{m}{(b-1) b^{i}}\right) b^{i \gamma} = \infty$$ by \eqref{eq:remark} and the monotonicity of $r^{-\gamma} f(r)$. So in this case, \cref{thm:A} is also valid. 
\end{remark} 

The proof of \cref{thm:A} naturally splits into two parts: the convergence part and the divergence part. We start with the convergence part, which involves finding covers of $W_{b,A}(\psi) \cap C(b,D)$ of arbitrarily small measure.
\begin{lem} \label{lem:conv}
If $$\sum\limits_{i \in I \colon \psi_{A}(i) > \frac{m}{(b-1) b^{ i }}} f\left(\psi_{A}(i) - \frac{m}{(b-1) b^{i}}\right) b^{i \gamma } < \infty,$$ then
\begin{align*}
\mathcal{H}^{f} (W_{b,A}(\psi) \cap C(b,D))= 0.
\end{align*}
\end{lem}
\begin{proof}
For $i \in I$, let
\begin{align*}
S_{i} &= \bigcup_{0 \leq p \leq b^{i}} B\left(\frac{p}{b^{i}},\psi_{A}(i)\right) \cap C(b,D).
\end{align*}
For each $0 \leq p \leq b^{i}$, we have 
$$B\left(\frac{p}{b^{i}},\psi_{A}(i)\right) = \bigcup_{n \colon a_{n}=i} B\left(\frac{p}{b^{a_{n}}},\psi(n)\right),$$
since all balls on the right side have the same center and $\psi_{A}(i)$ is the maximum of their radii. So
\begin{align*}
S_{i} = \bigcup_{a_{n}=i} \bigcup_{0 \leq p \leq b^{i}} B\left(\frac{p}{b^{a_{n}}},\psi(n)\right) \cap C(b,D).
\end{align*}
Therefore 
\begin{align*}
W_{b,A}(\psi) \cap C(b,D) &= \limsup_{n \to \infty} \bigcup_{0 \leq p \leq b^{a_{n}}} B\left(\frac{p}{b^{a_{n}}},\psi(n)\right) \cap C(b,D) \\
&= \limsup_{i \to \infty} S_{i}.
\end{align*} 
Lemma \ref{lem:intersection} implies that $S_{i}$ is a subset of 
\begin{align*}
\bigcup_{\frac{p}{b^{i}} \in L_{i} \cup R_{i}} B\left(\frac{p}{b^{i}} + d_{l,i},\psi_{A}(i)-d_{l,i}\right) \cup B\left(\frac{p}{b^{i}} - d_{r,i},\psi_{A}(i)-d_{r,i}\right),
\end{align*}
where $$d_{l,i} = \frac{m_{l}}{(b-1)b^i} \quad \text{ and } \quad d_{r,i} = \frac{m_{r}}{(b-1)b^i}.$$
Recall that $m = \min\{m_{l}, m_{r}\}$, so $$f(\psi_{A}(i)-d_{l,i}) + f(\psi_{A}(i)-d_{r,i}) \leq 2 f\left(\psi_{A}(i)-\frac{m}{(b-1)b^i}\right).$$
Then for any $\rho > 0$ and any integer $i_{0}$ big enough, we have 
\begin{align*}
& \mathcal{H}^{f}_{\rho}(W_{b,A}(\psi) \cap C(b,D)) \\
\ll& \sum_{i \geq i_{0} \colon i \in I} \mathcal{H}^{f}_{\rho}(S_{i}) \\
\ll& \sum_{i \geq i_{0} \colon i \in I, \psi_{A}(i) > \frac{m}{(b-1)b^i}} f\left(\psi_{A}(i)-\frac{m}{(b-1)b^i}\right) \times \#(L_{i} \cup R_{i}) \\
\ll& \sum_{i \geq i_{0} \colon i \in I, \psi_{A}(i) > \frac{m}{(b-1)b^i}} f\left(\psi_{A}(i)-\frac{m}{(b-1)b^i}\right) \times b^{i \gamma}.
\end{align*}
Let $i_{0} \to \infty$ and then $\rho \to 0$, we deduce $\mathcal{H}^{f}(W_{b,A}(\psi) \cap C(b,D)) = 0$.
\end{proof}

Now we turn to the more difficult divergence part. We are going to use Lemma \ref{lem:intersection} to rewrite balls $B(p b^{-n},\psi(n))$ as balls with center in $C(b,D)$. Note that the new balls could have different radii depending on whether $p b^{-n}$ is a left or right endpoint of $C_{n}(b,D)$, and we will deal with these two cases separately. For any $i \geq 1$, let
\begin{align*}
L_{i}^{\ast} &= \left\{\frac{p}{b^{i}} \in L_{i} \colon \frac{p}{b^{i}} + d_{l,i} \neq \frac{q}{b^{j}} + d_{l,j} \text{ for any } q \text{ and } j<i \right\}, \\
LS_{i}^{\ast} &= \bigcup_{\frac{p}{b^{i}} \in L_{i}^{\ast}} B\left(\frac{p}{b^{i}}+d_{l,i},\psi_{A}(i) - d_{l,i}\right),
\end{align*}
and $LW_{b,A}^{\ast}(\psi) = \limsup_{i \to \infty} LS_{i}^{\ast}$. Replacing $L_{i}$ by $R_{i}$, the set $RW_{b,A}^{\ast}(\psi)$ is defined in a similar way, and Lemma \ref{lem:intersection} implies that
\begin{align} \label{eq:union}
LW_{b,A}^{\ast}(\psi) \cup RW_{b,A}^{\ast}(\psi) \subset W_{b,A}(\psi).
\end{align}

Let $\mu = \mathcal{H}^{\gamma}\vert_{C(b,D)}$, $B$ be an arbitrary ball with center in $C(b,D)$ and 
\begin{align*}
LS_{i}^{\ast}(B) &= B \cap LS_{i}^{\ast}.
\end{align*}
Recall that $\mu$ is $\gamma$-Ahlfors regular, so there exists a constant $r_{0}>0$ such that for any ball $B(x,r_{1})$ with $x \in C(b,D)$ and $r_{1}<r_{0}$, we have 
\begin{align} \label{eq:ball}
\mu(B(x,r_{1})) \asymp r_{1}^{\gamma}.
\end{align}

Suppose the radius of $B$ satisfies that $r(B) < r_{0}/2$, so (\ref{eq:ball}) implies that $\mu(2B) \asymp \mu(B)$. For ease of notation, we write
\begin{align*}
B_{i}^{\ast}(\psi_{A}) :=& B\left(\frac{p}{b^{i}} + d_{l,i},\psi_{A}(i)-d_{l,i}\right) \hspace{1pt} \text{ and } \hspace{1pt} B_{i}^{\ast} := B\left(\frac{p}{b^{i}} + d_{l,i},\frac{1}{2 b^{i}}\right)
\end{align*} 
when the value of $p$ is unimportant. Then
\begin{align*}
&\#\{B^{\ast}_{i}(\psi_{A}) \subset B \colon B^{\ast}_{i}(\psi_{A}) \cap C(b,D) \neq \emptyset\} \\
\asymp &\#\{B^{\ast}_{i} \subset B \colon B^{\ast}_{i} \cap C(b,D) \neq \emptyset\} \\
\asymp &\frac{\mu(B)}{\mu(B^{\ast}_{i})} \quad \text{ since $B^{\ast}_{i}$ are disjoint,} \\
\asymp &\mu(B)b^{i \gamma}.
\end{align*}
Therefore
\begin{align} \label{eq:mu_LS}
\mu (LS_{i}^{\ast}(B)) \asymp \mu(B) b^{i \gamma} \mu(B^{\ast}_{i}(\psi_{A})) \asymp \mu(B) b^{i \gamma} (\psi_{A}(i)-d_{l,i})^{\gamma}.
\end{align}

Next we show that the $\mu(LS_{i}^{\ast}(B))$ satisfies a quasi-independence relation.

\begin{lem}\label{lem:quasi-independence}
Suppose $\psi_{A}(i) \leq b^{-i}/2$ for all $i$. Let $t_{0}$ be a sufficiently large integer satisfying $b^{-t_{0}} < r(B)$. Then there exists a constant $K>0$ such that for any $i>j>t_{0}$, 
\begin{align*}
\mu(LS_{i}^{\ast}(B) \cap LS_{j}^{\ast}(B)) \leq \frac{K}{\mu(B)} \mu(LS_{i}^{\ast}(B)) \mu(LS_{j}^{\ast}(B)).
\end{align*}
\end{lem}
\begin{proof}
We first consider the case $\psi_{A}(j) - d_{l,j} \leq b^{-i}/2$. Let $p_{1} b^{-i} \in L_{i}^{\ast}$ and $p_{2} b^{-j} \in L_{j}^{\ast}$. By the definition of $L_{i}^{\ast}$, the two ball centers $p_{1} b^{-i} + d_{l,i}$ and $p_{2} b^{-j} + d_{l,j}$ are distinct. Recall that $$d_{l,i} = \frac{m_{l}}{(b-1)b^{i}},$$ so the distance between the two centers is 
\begin{align*}
\abs{\frac{p_{1}}{b^{i}}+d_{l,i} - \frac{p_{2}}{b^{j}}-d_{l,j}} &= \abs{\frac{p_{1}}{b^{i}}-\frac{m_{l}}{(b-1)b^{i}} - \frac{p_{2}}{b^{j}}+\frac{m_{l}}{(b-1)b^{j}}} \\
&= \abs{\frac{p_{1} - p_{2} b^{i-j}}{b^{i}} - \frac{m_{l} (b^{i-j} - 1)}{(b-1)b^{j}}} \\
&\geq \frac{1}{b^{i}}
\end{align*}
since $b-1$ divides $b^{i-j} - 1$. Note that both radii $\psi_{A}(i) - d_{l,i}$ and $\psi_{A}(j) - d_{l,j}$ are not greater than $b^{-i}/2$, hence
$$B\left(\frac{p_{1}}{b^{i}} + d_{l,i},\psi_{A}(i)-d_{l,i}\right) \cap B\left(\frac{p_{2}}{b^{j}}+d_{l,j},\psi_{A}(j)-d_{l,j}\right) = \emptyset.$$
Therefore
$$LS_{i}^{\ast}(B) \cap LS_{j}^{\ast}(B) = \emptyset$$
and thus $$\mu(LS_{i}^{\ast}(B) \cap LS_{j}^{\ast}(B)) = 0 \leq \frac{K}{\mu(B)} \mu(LS_{i}^{\ast}(B)) \mu(LS_{j}^{\ast}(B))$$ for any constant $K>0$.

Now assume $\psi_{A}(j) - d_{l,j} > b^{-i}/2$. We have
\begin{align}
&\mu(LS_{i}^{\ast}(B) \cap LS_{j}^{\ast}(B)) \nonumber \\
= &\mu \left(LS_{i}^{\ast}(B) \cap B \cap \bigcup_{b p^{-j} \in L_{j}^{\ast}} B\left(\frac{p}{b^{j}},\psi_{A}(j)\right)\right) \nonumber\\
\leq &\mathcal{N}(j) \mu\left(LS_{i}^{\ast}(B) \cap B_{j}^{\ast}(\psi_{A})\right), \label{eq:quasi_1}
\end{align}
where $$\mathcal{N}(j) = \#\left\{B_{j}^{\ast}(\psi_{A}) \colon B_{j}^{\ast}(\psi_{A}) \cap B \cap C(b,D) \neq \emptyset\right\}.$$
Let $2B$ denotes the ball with same center as $B$ but twice the radius. Since $2\psi_{A}(j) \leq b^{-j} \leq b^{-t_{0}} < r(B)$, we have 
\begin{align}
\mathcal{N}(j) \leq& \#\left\{B_{j}^{\ast}(\psi_{A}) \subset 2B \colon B_{j}^{\ast}(\psi_{A}) \cap C(b,D) \neq \emptyset\right\} \nonumber\\
\leq&  \#\left\{B_{j}^{\ast} \subset 2B \colon B_{j}^{\ast} \cap C(b,D) \neq \emptyset\right\} \nonumber\\
\leq& \frac{\mu(2B)}{\mu(B_{j}^{\ast})} \quad \text{because $B_{j}^{\ast}$ are disjoint,} \nonumber\\
\ll& \mu(B) b^{j \gamma} \quad \text{ (by $\gamma$-Ahlfors regularity)}. \label{eq:quasi_2}
\end{align}
Similarly, for any fixed $j$,
\begin{align*}
& \#\left\{B_{i}^{\ast}(\psi_{A}) \colon B_{i}^{\ast}(\psi_{A}) \cap B_{j}^{\ast}(\psi_{A}) \cap C(b,D) \neq \emptyset\right\} \\
\leq& \#\left\{B_{i}^{\ast} \colon B_{i}^{\ast} \cap B_{j}^{\ast}(\psi_{A}) \cap C(b,D) \neq \emptyset\right\} \\
\leq& 2 + \frac{\mu(B_{j}^{\ast}(\psi_{A}))}{\mu(B_{i}^{\ast})} \\
\ll& 2 + (\psi_{A}(j)-d_{l,j})^{\gamma} b^{i \gamma},
\end{align*}
where the number $2$ is for the possible existence of those $B_{i}^{\ast}$ which intersect with but are not contained in $B_{j}^{\ast}(\psi_{A})$. Then
\begin{align}
&\mu\left(LS_{i}^{\ast}(B) \cap B_{j}^{\ast}(\psi_{A})\right) \nonumber \\
\ll &\mu(B_{i}^{\ast}(\psi_{A})) (2 + (\psi_{A}(j)-d_{l,j})^{\gamma} b^{i \gamma}) \nonumber\\
\ll &(\psi_{A}(i)-d_{l,i})^{\gamma} + (\psi_{A}(i)-d_{l,i})^{\gamma} (\psi_{A}(j)-d_{l,j})^{\gamma} b^{i \gamma} \nonumber\\
\ll &(\psi_{A}(i)-d_{l,i})^{\gamma} (\psi_{A}(j)-d_{l,j})^{\gamma} b^{i \gamma} \label{eq:quasi_3}
\end{align}
since $(\psi_{A}(j)-d_{l,j}) b^{i} > 1/2$. 

Now \eqref{eq:mu_LS}, \eqref{eq:quasi_1}, \eqref{eq:quasi_2} and \eqref{eq:quasi_3} give that
\begin{align*}
\mu(LS_{i}^{\ast}(B) \cap LS_{j}^{\ast}(B)) \ll& \mu(B) b^{j \gamma} (\psi_{A}(i)-d_{l,i})^{\gamma} (\psi_{A}(j)-d_{l,j})^{\gamma} b^{i \gamma} \\
\ll& \frac{\mu(LS_{i}^{\ast}(B)) \mu(LS_{j}^{\ast}(B))}{\mu(B)}.
\end{align*}
\end{proof}

\begin{prop} \label{mu_case}
Let $\mu = \mathcal{H}^{\gamma}\vert_{C(b,D)}$. If $\psi_{A}(i) \leq b^{-i}/2$ for all $i \in I$ and
$$\sum\limits_{i \in I \colon \psi_{A}(i) > d_{l,i}} \left(\psi_{A}(i) - d_{l,i}\right)^{\gamma} b^{i \gamma} = \infty,$$
then
\begin{align*}
\mu(LW_{b,A}^{\ast}(\psi)) = \mu(C(b,D)).
\end{align*}
\end{prop}

\begin{proof}
Let $t_{0}$ be the number as in Lemma \ref{lem:quasi-independence}. For all $i>t_{0}$, \eqref{eq:Div_BC}, \eqref{eq:mu_LS} and Lemma \ref{lem:quasi-independence} imply that $$\mu (\limsup_{i \to \infty, i \in I} S_{i}^{\ast}(B)) \geq \frac{\mu(B)}{C}.$$
Applying Lemma \ref{lem:full_measure} and noting that $$\limsup_{i \to \infty, i \in I} LS_{i}^{\ast}(B) = B \cap \limsup_{i \to \infty, i \in I} LS_{i}^{\ast} = B \cap LW_{b,A}^{\ast}(\psi),$$
we have $\mu(LW_{b,A}^{\ast}(\psi)) = \mu(C(b,D))$.
\end{proof}

Now we extends to $f$-Hausdorff measure by applying the mass transference principle (\cref{thm:MTP}). 

\begin{lem} \label{lem:left}
Let $f$ be a dimension function such that $r^{-\gamma} f(r)$ is monotonic. If $\psi_{A}(i) < b^{-i}/2$ for all $i \in I$ and
$$\sum\limits_{i \in I \colon \psi_{A}(i) > d_{l,i}} f\left(\psi_{A}(i) - d_{l,i}\right) b^{i \gamma} = \infty,$$
then
\begin{align*}
\mathcal{H}^{f}(LW_{b,A}^{\ast}(\psi) \cap C(b,D)) = \mathcal{H}^{f}(C(b,D)).
\end{align*}
\end{lem}
\begin{proof}
Define a function $\theta: \N \to (0,\infty)$ by
\begin{align*}
\theta(n) = \begin{cases}
f(\psi(n) - d_{l,a_{n}})^{1/\gamma} + d_{l,a_{n}}, &\text{ if } \psi(n) > d_{l,a_{n}},\\
d_{l,a_{n}}/2, &\text{ otherwise.}
\end{cases}
\end{align*}
Let $$\theta_{A}(i) = \max \left\{\theta(n) \colon a_{n}=i \right\},$$ then
$$\sum_{i \in I \colon \theta_{A}(i) > d_{l,i}} \left(\theta_{A}(i)-d_{l,i}\right)^{\gamma} \times b^{i\gamma} = \sum_{i \in I \colon \psi_{A}(i) > d_{l,i}} f(\psi_{A}(i) - d_{l,i}) \times b^{i\gamma} = \infty.$$
Now Proposition \ref{mu_case} says that
$$\mu (LW_{b,A}^{\ast}(\theta)) = \mu(C(b,D)), $$
which is equivalent to 
$$\mathcal{H}^{\gamma} (LW_{b,A}^{\ast}(\theta) \cap C(b,D)) = \mathcal{H}^{\gamma}(C(b,D)).$$
Hence \cref{thm:MTP} implies that
$$\mathcal{H}^{f} (LW_{b,A}^{\ast}(\psi) \cap C(b,D)) = \mathcal{H}^{f}(C(b,D)).$$
\end{proof}

A similar result for $RW_{b,A}^{\ast}(\psi)$ can be proved by the same method.
\begin{lem} \label{lem:right}
Let $f$ be a dimension function such that $r^{-\gamma} f(r)$ is monotonic. If $\psi_{A}(i) < b^{-i}/2$ for all $i$ and
$$\sum\limits_{i \in I \colon \psi_{A}(i) > d_{r,i}} f\left(\psi_{A}(i) - d_{r,i}\right) b^{i \gamma} = \infty,$$
then
\begin{align*}
\mathcal{H}^{f}(RW_{b,A}^{\ast}(\psi)) = \mathcal{H}^{f}(C(b,D)).
\end{align*}
\end{lem}

Note that $m = \min\{m_{l}, m_{r}\}$, so $$\sum\limits_{i \in I \colon \psi_{A}(i) > \frac{m}{(b-1) b^{i }}} f\left(\psi_{A}(i) - \frac{m}{(b-1) b^{i}}\right) b^{i \gamma} = \infty,$$ implies either
\begin{align*}\sum\limits_{i \in I \colon \psi_{A}(i) > d_{l,i}} f\left(\psi_{A}(i) - d_{l,i}\right) b^{i \gamma} = \infty,\end{align*}
or
\begin{align*}
\sum\limits_{i \in I \colon \psi_{A}(i) > d_{r,i}} f\left(\psi_{A}(i) - d_{r,i}\right) b^{i \gamma} = \infty.
\end{align*}

Therefore the divergence part of \cref{thm:A} is a consequence of Lemma \ref{lem:left} and Lemma \ref{lem:right} since $LW_{b,A}^{\ast}(\psi) \cup RW_{b,A}^{\ast}(\psi) \subset W_{b,A}(\psi).$

\section{Proof of \cref{thm:main}} \label{sec:dependent}
For any prime divisor $q$ of $b$, we have
$$v_{q} (b^{\lfloor \alpha_{1} a_n \rfloor}) = \lfloor \alpha_{1} a_n \rfloor v_{q}(b) \leq \alpha_{1} a_n v_{q}(b) \leq a_n v_{q}(t) = v_{q}(t^{a_n}).$$
Hence $b^{\lfloor \alpha_{1} a_n \rfloor} \mid t^{a_{n}}$ and thus
\begin{align*} 
\bigcup_{0 \leq p \leq b^{\lfloor \alpha_{1} a_n \rfloor}} B\left(\frac{p}{b^{\lfloor \alpha_{1} a_n \rfloor}},\psi(n)\right) \subset \bigcup_{0 \leq p \leq t^{a_n}} B\left(\frac{p}{t^{a_n}},\psi(n)\right).
\end{align*}
Therefore
\begin{align} \label{eq:contain_main_l}
W_{b,(\lfloor \alpha_{1} a_n \rfloor)_{n \geq 1}}(\psi) \subset W_{t,(a_{n})_{n \geq 1}}(\psi).
\end{align}

Let $J_{1} = I(\{\lfloor i \alpha_{1} \rfloor\}_{i \in I})$ be the set of integers appearing in the sequence $\{\lfloor i \alpha_{1} \rfloor\}_{i \in I}$. Note that
\begin{align*}
\psi_{(\lfloor a_{n} \alpha_{1} \rfloor)_{n \geq 1}}(j) &= \max\left\{\psi(n) \colon \lfloor a_{n} \alpha_{1} \rfloor = j \right\} \\
&=\max\left\{\psi_{A}(i) \colon \lfloor i \alpha_{1} \rfloor = j \right\}.
\end{align*}
Then
\begin{align*}
& \sum\limits_{i \in I \colon \psi_{A}(i) > \frac{m}{(b-1) b^{\lfloor i \alpha_{1} \rfloor }}} f\left(\psi_{A}(i) - \frac{m}{(b-1) b^{\lfloor i \alpha_{1} \rfloor }}\right) b^{i \alpha_{1} \gamma} \\
= & \sum\limits_{j \in J_{1}} \sum_{\substack{i \colon \lfloor i \alpha_{1} \rfloor = j \\ \psi_{A}(i) > \frac{m}{(b-1) b^{\lfloor i \alpha_{1} \rfloor }}}} f\left(\psi_{A}(i) - \frac{m}{(b-1) b^{\lfloor i \alpha_{1} \rfloor }}\right) b^{i \alpha_{1} \gamma} \\
\leq &\sum_{\substack{j \in J_{1} \\ \psi_{(\lfloor a_{n} \alpha_{1} \rfloor)_{n \geq 1}}(j) > \frac{m}{(b-1) b^{j}}}} \sum_{i \colon \lfloor i \alpha_{1} \rfloor = j} f\left(\psi_{(\lfloor a_{n} \alpha_{1} \rfloor)_{n \geq 1}}(j) - \frac{m}{(b-1) b^{\lfloor i \alpha_{1} \rfloor }}\right) b^{i \alpha_{1} \gamma} \\
\ll &\sum_{j \in J_{1} \colon \psi_{(\lfloor a_{n} \alpha_{1} \rfloor)_{n \geq 1}}(j) > \frac{m}{(b-1) b^{j}}} f\left(\psi_{(\lfloor a_{n} \alpha_{1} \rfloor)_{n \geq 1}}(j)  - \frac{m}{(b-1) b^{j}}\right) b^{ j \gamma}.
\end{align*}

Therefore \cref{thm:A} and \eqref{eq:contain_main_l} imply that
\begin{align*}
\mathcal{H}^{f} (W_{t}(\psi) \cap C(b,D)) = \mathcal{H}^{f}(C(b,D))
\end{align*}
if
\begin{align*}
\sum\limits_{i \in I \colon \psi_{A}(i) > \frac{m}{(b-1) b^{\lfloor i \alpha_{1} \rfloor }}} f\left(\psi_{A}(i) - \frac{m}{(b-1) b^{\lfloor i \alpha_{1} \rfloor }}\right) b^{i \alpha_{1} \gamma} = \infty.
\end{align*}

The other half of the theorem is proved similarly using $\lceil \alpha_{2} a_n \rceil$. For any prime divisor $q$ of $b$, we have
$$v_{q} (b^{\lceil \alpha_{2} a_n \rceil}) = \lceil \alpha_{2} a_n \rceil v_{q}(b) \geq \alpha_{2} a_n v_{q}(b) \geq a_{n} v_{q}(t) = v_{q}(t^{a_{n}}),$$
which implies $t^{a_{n}} \mid b^{\lceil \alpha_{2} a_n \rceil}$. Therefore
\begin{align*} 
\bigcup_{0 \leq p \leq t^{a_n}} B\left(\frac{p}{t^{a_n}},\psi(a_n)\right) \subset \bigcup_{0 \leq p \leq b^{\lceil \alpha_{2} a_n \rceil}} B\left(\frac{p}{b^{\lceil \alpha_{2} a_n \rceil}},\psi(a_n)\right)
\end{align*}
and hence
\begin{align} \label{eq:contain_main_r}
W_{t,(a_{n})_{n \geq 1}}(\psi) \subset W_{b,(\lfloor \alpha_{2} a_n \rfloor)_{n \geq 1}}(\psi).
\end{align}
Let $J_{2} = I(\{\lceil i \alpha_{2} \rceil\}_{i \in I})$ be the set of integers appearing in $\{\lceil i \alpha_{2} \rceil\}_{i \in I}$. Then
\begin{align*}
& \sum\limits_{i \in I \colon \psi_{A}(i) > \frac{m}{(b-1) b^{\lceil i \alpha_{2} \rceil }}} f\left(\psi_{A}(i) - \frac{m}{(b-1) b^{\lceil i \alpha_{2} \rceil}}\right) b^{i \alpha_{2} \gamma} \\
= & \sum\limits_{j \in J_{2}} \sum_{\substack{i \colon \lceil i \alpha_{2} \rceil = j \\ \psi_{A}(i) > \frac{m}{(b-1) b^{\lceil i \alpha_{2} \rceil}}}} f\left(\psi_{A}(i) - \frac{m}{(b-1) b^{\lceil i \alpha_{2} \rceil}}\right) b^{i \alpha_{2} \gamma} \\
\geq &\sum_{j \in J_{2} \colon \psi_{(\lceil a_{n} \alpha_{2} \rceil)_{n \geq 1}}(j) > \frac{m}{(b-1) b^{j}}} f\left(\psi_{(\lceil a_{n} \alpha_{2} \rceil)_{n \geq 1}}(j) - \frac{m}{(b-1) b^{\lceil i \alpha_{2} \rceil }}\right) b^{i \alpha_{2} \gamma} \\
\gg &\sum_{j \in J_{2} \colon \psi_{(\lceil a_{n} \alpha_{2} \rceil)_{n \geq 1}}(j) > \frac{m}{(b-1) b^{j}}} f\left(\psi_{(\lceil a_{n} \alpha_{2} \rceil)_{n \geq 1}}(j) - \frac{m}{(b-1) b^{j}}\right) b^{j \gamma} \\
\end{align*}

Now \cref{thm:A} and \eqref{eq:contain_main_r} imply that
\begin{align*}
\mathcal{H}^{f} (W_{t}(\psi) \cap C(b,D)) = 0
\end{align*} 
if
\begin{align*}
\sum\limits_{i \in I \colon \psi_{A}(i) > \frac{m}{(b-1) b^{\lceil i \alpha_{2} \rceil }}} f\left(\psi_{A}(i) - \frac{m}{(b-1) b^{\lceil i \alpha_{2} \rceil}}\right) b^{i \alpha_{2} \gamma} < \infty.
\end{align*} 

\section{Further discussion} \label{sec:further}
\subsection{Hausdorff dimension} \label{sec:dimension}
When $b$ and $t$ are multiplicatively dependent, we have computed the Hausdorff $f$-measure of $W_{t}(\psi) \cap C(b,D)$ for any dimension function such that $r^{-\gamma} f(r)$ is monotonic (\cref{thm:dependent}). In particular, if the set $D$ contains at least one of $0$ and $b-1$, take $f(r) = r^{s}$ for any $s \geq 0$ and we have
\begin{align*}
\mathcal{H}^{s} (W_{t}(\psi) \cap C(b,D)) = \begin{cases}
0, &\text{ if } \sum_{n=1}^{\infty} \psi(n)^{s} \times t^{n \gamma} < \infty, \\
\mathcal{H}^{s}(C(b,D)), &\text{ if } \sum_{n=1}^{\infty}  \psi(n)^{s} \times t^{n \gamma} = \infty.
\end{cases}
\end{align*}
Note that
\begin{align*}
\sum_{n=1}^{\infty} \psi(n)^{s} \times t^{n \gamma} = \sum_{n=1}^{\infty} t^{n\left(\gamma + s\frac{\log \psi(n)}{n \log t}\right)}.
\end{align*}
Since $t \geq 2$, the above series converges if $$s > \limsup_{n \to \infty} \frac{-\gamma n \log t }{\log \psi(n)}$$ 
and diverges if $$s < \limsup_{n \to \infty} \frac{-\gamma n \log t }{\log \psi(n)}.$$ 
Let $$\lambda_{\psi} = \liminf_{n \to \infty} \frac{-\log \psi(n)}{n \log t}.$$ For the rest of this article, we assume $\lambda_{\psi} \geq 1$ as otherwise the situation is trivial. By computations above, we deduce the Hausdorff dimension of $W_{t}(\psi) \cap C(b,D)$.

\begin{cor} \label{cor:dimension_dependent}
Suppose $b$ and $t$ are multiplicatively dependent, $D$ contains at least one of $0$ and $b-1$ and $\lambda_{\psi} \geq 1$. Then
$$\dim_{\rm H} W_{t}(\psi) \cap C(b,D) = \frac{\gamma}{\lambda_{\psi}}.$$
\end{cor}

\subsection{Multiplicatively independent case}
We first consider the situation that $b$ and $t$ have the same prime divisors but not necessarily multiplicatively dependent. Using \cref{thm:main}, we deduce that
\begin{align*}
\mathcal{H}^{s} (W_{t}(\psi) \cap C(b,D)) = \begin{cases}
0, &\text{ if } \sum_{n=1}^{\infty} \psi(n)^{s} \times b^{n \alpha_{2} \gamma} < \infty, \\
\mathcal{H}^{s}(C(b,D)), &\text{ if } \sum_{n=1}^{\infty}  \psi(n)^{s} \times b^{n \alpha_{1} \gamma} = \infty
\end{cases}
\end{align*}
if $D$ contains at least one of $0$ and $b-1$. Then a similar computation as in the previous section yields
\begin{cor} \label{cor:dim_same_prime}
Suppose $b$ and $t$ have the same prime divisors and $D$ contains at least one of $0$ and $b-1$. If $\lambda_{\psi} \geq 1$, then
$$\frac{\alpha_{1} \log b}{\log t} \frac{\gamma}{\lambda_{\psi}} \leq \dim_{\rm H} W_{t}(\psi) \cap C(b,D) \leq \frac{\alpha_{2} \log b}{\log t} \frac{\gamma}{\lambda_{\psi}}.$$
\end{cor}

Based on Corollary \ref{cor:dimension_dependent} and note that $$\frac{\alpha_{1} \log b}{\log t} \leq 1 \leq \frac{\alpha_{2} \log b}{\log t},$$ we propose the following conjecture.
\begin{conj} \label{conj:dim_same_prime}
Suppose $b$ and $t$ have the same prime divisors, $D$ contains at least one of $0$ and $b-1$ and $\lambda_{\psi} \geq 1$, then
$$\dim_{\rm H} W_{t}(\psi) \cap C(b,D)= \frac{\gamma}{\lambda_{\psi}}.$$
\end{conj}

In \cite{ACY20}, the authors gave a heuristic of Conjecture \ref{conj:Velani}. Here we modify their argument to formulate a conjecture for $\mathcal{H}^{f} (W_{t}(\psi) \cap C(b,D))$ when the sets of prime divisors of $b$ and $t$ are different and $D$ contains at least one of $0$ and $b-1$. For any big integer $n$, choose another integer $m$ such that $b^{-m} \asymp \psi(n)$. Divide $[0,1]$ into $b^m$ intervals of length $b^{-m}$, among them there are approximately $b^{m \gamma}$ intervals that intersect with $C(b,D)$. If the sets of prime divisors of $b$ and $t$ are different, the distribution of points $p t^{-n}$ should be random with respect to those length $b^{-m}$ intervals. So in the union
$$\bigcup_{0 \leq p \leq t^{n}} B\left(\frac{p}{t^n}, \psi(n)\right),$$
there are about $t^n b^{m (\gamma - 1)}$ balls that intersect with $C(b,D)$. Therefore
\begin{align*}
\mathcal{H}^{f} (W_{t}(\psi) \cap C(b,D)) \ll& \sum_{n=n_{0}}^{\infty} f(\psi(n)) t^{n}  b^{m (\gamma - 1)} \\
\ll& \sum_{n=n_{0}}^{\infty} f(\psi(n)) t^{n} \psi(n)^{1-\gamma}
\end{align*}
for any $n_{0} > 0$. Based on above heuristic argument, we propose the following conjecture.
\begin{conj}  \label{conj:indepedent}
Let $C(b,D)$ be a missing digit set with Hausdorff dimension $\gamma$, the set $D$ contains at least one of $0$ and $b-1$, $t \geq 2$ be an integer satisfying the sets of prime divisors of $b$ and $t$ are different, $f$ be a dimension function such that $r^{-\gamma} f(r)$ is monotonic, and $\psi : \N \to (0,\infty)$ be a function. Then
\begin{align*}
\mathcal{H}^{f} (W_{t}(\psi) \cap C(b,D)) = \begin{cases}
0, \hspace{-0.45cm} &\text{ if } \sum\limits_{n=1}^{\infty} f(\psi(n)) t^{n} \psi(n)^{1-\gamma} < \infty, \\
\mathcal{H}^{f}(C(b,D)), \hspace{-0.45cm} &\text{ if } \sum\limits_{n=1}^{\infty}  f(\psi(n)) t^{n} \psi(n)^{1-\gamma}= \infty.
\end{cases}
\end{align*}
\end{conj}

Assuming Conjecture \ref{conj:indepedent}, the Hausdorff dimension of $W_{t}(\psi) \cap C(b,D)$ can be computed in the same way as we did in \cref{sec:dimension}, so the following conjecture holds if Conjecture \ref{conj:indepedent} is valid.
\begin{conj} \label{conj:dimension_independent}
Let $C(b,D)$ be a missing digit set with Hausdorff dimension $\gamma$, with $D$ containing at least one of $0$ and $b-1$, $t \geq 2$ be an integer satisfying the sets of prime divisors of $b$ and $t$ are different. If $\lambda_{\psi} \geq 1$, then
\begin{equation}
\dim_{\rm H} (W_{t}(\psi) \cap C(b,D)) = \max\left\{\frac{1}{\lambda_{\psi}} + \gamma - 1, 0 \right\}. \label{eq:dimension_independent}
\end{equation}
\end{conj}

When $b=3$ and $t$ is not a power of $3$, formula \eqref{eq:dimension_independent} was already conjectured by Bugeaud and Durand \cite[Conjecture 1.2]{BD16}. Note that this formula is very different with the multiplicatively dependent case Corollary \ref{cor:dimension_dependent} and the same prime divisors case Corollary \ref{cor:dim_same_prime}. It is known that $\dim_{\rm H} (W_{t}(\psi)) = 1/\lambda_{\psi}$ (see \cite{HV95}), so Corollary \ref{cor:dimension_dependent} says that the Hausdorff dimension of the intersection $W_{t}(\psi) \cap C(b,D)$ is the product of two dimensions when $b$ and $t$ are multiplicatively dependent, while Conjecture \ref{conj:dimension_independent} predicts that the Hausdorff dimension of the intersection is the sum of two dimensions minus one when the set of prime divisors of $b$ and $t$ are different. For two unrelated fractals, the Hausdorff dimension of their intersection is likely equal to the sum of dimensions minus the dimension of the ambient space (see \cite[Chapter 8]{Falconer90}), so Conjecture \ref{conj:dimension_independent} is consistent with the intuition that $W_{t}(\psi)$ and $C(b,D)$ are \enquote{unrelated} when the sets of prime divisors of $b$ and $t$ are different. This kind of formula also appears in \cite{HLX22,WWX17} as well as in the study of other limsup sets induced by recurrence of orbits in dynamical systems \cite{CWW19,HLSW22}. Note that if $b$ and $t$ are multiplicatively independent but have the same prime divisors, then Corollary \ref{cor:dim_same_prime} implies that
$$\dim_{\rm H} W_{t}(\psi) \cap C(b,D) \quad \text{ and } \quad \max\left\{\frac{1}{\lambda_{\psi}} + \gamma - 1, 0 \right\}$$ 
are not necessarily equal since the right hand side becomes $0$ when $\lambda_{\psi}$ is large while the lower bound in Corollary \ref{cor:dim_same_prime} is always positive.

\subsection{Large intersection property}
We are also able to show the large intersection property of $W_{t}(\psi) \cap C(b,D)$ when $\psi(n) = t^{-\theta n}$ for some $\theta > 1$. The large intersection property was introduced by Falconer \cite{Falconer94} and has many applications, we refer to \cite{HLY22} and references therein. Let $(X,\mathscr{B},\mu,d)$ be a compact metric space such that $\mu$ is $\gamma$-Ahlfors regular. When $X = \R^d$, the set $\mathcal{G}^{s}(X)$ is defined to be the class of all $G_{\delta}$ sets $F$ in $X$ such that $$ \dim_{\rm H} \cap_{n=1}^{\infty} g_{n}(F) \geq s$$ holds for all sequences of similarity transformations $\{g_{n}\}_{n \geq 1}$. Definition of $\mathcal{G}^{s}(X)$ for general metric spaces is more complicated and can be found in \cite{NS22}. For any Borel set $U$, define $$I_{s}(\mu,U) = \int_{U} \int_{U} \abs{x-y}^{-s} d\mu(x) d\mu(y).$$

Next we state two results that enable us to deduce the large intersection property of $W_{t}(\psi) \cap C(b,D)$.

\begin{thm}[{\cite[Theorem 1.1]{HLY22}}] \label{thm:large_intersection}
Suppose $\{B_{n}\}_{n \geq 1}$ is a sequence of balls in $X$ whose radii decrease to $0$ as $n \to \infty$. For each $n \geq 1$, let $E_{n}$ be an open subset of $B_{n}$ and define
$$\lambda = \sup \left\{s \geq 0 \colon \sup_{n \geq 1} \frac{\mu(B_{n}) I_{s}(\mu,E_{n})}{\mu(E_{n})^2} < \infty \right\}.$$
Then $\mu(\limsup_{n \to \infty} B_{n})  = \mu(X)$ implies $\limsup_{n \to \infty} E_{n} \in \mathcal{G}^{\lambda}(X)$.
\end{thm}

\begin{lem}[{\cite[Lemma 3.1]{HLY22}}] \label{lem:large_intersection}
Suppose $0 \leq s < \gamma$ and $U$ is a Borel set with diameter $\diam(U) > 0$. Then
$$ I_{s}(\mu,U) \ll \diam(U)^{\gamma - s} \mu(U).$$
\end{lem}

In our setting, we take $X = C(b,D)$, $\mu = \mathcal{H}^{\gamma} \vert_{C(b,D)}$ and obtain the following result.
\begin{cor}
Suppose the set $D$ contains at least one of $0$ and $b-1$, $\psi(n) = t^{-\theta n}$ for some $\theta > 1$, $s < \gamma/\theta$ and $\log t/ \log b = \alpha \in \Q$, then
$$W_{t}(\psi) \in \mathcal{G}^{s}(C(b,D)).$$
\end{cor}
\begin{proof}
Since $\theta > 1$, we have $\theta \alpha n \geq \lceil \alpha n \rceil$ for $n$ big enough. Hence
\begin{align*}
\psi(n) = \frac{1}{t^{\theta n}} = \frac{1}{b^{\theta \alpha n}} \leq \frac{1}{b^{\lceil \alpha n \rceil}}
\end{align*}
for $n$ big enough. Then for all $n$ big enough and any $0 \leq p \leq t^n$, we have $$B\left(\frac{p}{t^{n}},\psi(n)\right) \subset B\left(\frac{q}{b^{\lceil \alpha n \rceil}},\frac{1}{b^{\lceil \alpha n \rceil}}\right)$$
for some $0 \leq q \leq b^{\lceil \alpha n \rceil}$, since $t^n \mid b^{\lceil \alpha n \rceil}$.

If $p t^{-n} = q b^{-\lceil \alpha n \rceil}\notin C(b,D)$, then both balls above are disjoint with $C(b,D)$. Otherwise, the fact that $\mu$ is $\gamma$-Ahlfors regular and $D$ contains at least one of $0$ or $b-1$ imply that for $n$ big enough, we have
\begin{align}
\mu\left(B\left(\frac{q}{b^{\lceil \alpha n \rceil}},\frac{1}{b^{\lceil \alpha n \rceil}}\right)\right) &\ll b^{- \gamma \lceil \alpha n \rceil} \ll t^{-\gamma n}, \label{eq:large_inter_1} \\
\mu\left(B\left(\frac{p}{t^{n}},\psi(n)\right)\right) &\asymp \psi(n)^{\gamma}. \label{eq:large_inter_2}
\end{align}
Then
\begin{align*}
& \mu\left(B\left(\frac{q}{b^{\lceil \alpha n \rceil}},\frac{1}{b^{\lceil \alpha n \rceil}}\right)\right) I_{s}\left(\mu, B\left(\frac{p}{t^{n}},\psi(n)\right)\right) \\
\ll & t^{-\gamma n} \psi(n)^{\gamma - s} \mu\left(B\left(\frac{p}{t^{n}},\psi(n)\right)\right) \quad \text{ by \eqref{eq:large_inter_1} and Lemma \ref{lem:large_intersection},} \\
\ll & t^{-\gamma n} \psi(n)^{- s} \mu\left(B\left(\frac{p}{t^{n}},\psi(n)\right)\right)^2 \quad \text{by \eqref{eq:large_inter_2},} \\
\ll &  t^{n (-\gamma + \theta s)} \mu\left(B\left(\frac{p}{t^{n}},\psi(n)\right)\right)^2.
\end{align*}
For any $s < \gamma / \theta$, we have $\lim_{n \to \infty} t^{n (-\gamma + \theta s)} = 0$. Note that 
$$\limsup_{n \to \infty, 0 \leq q \leq b^{\lceil \alpha n \rceil}} B\left(\frac{q}{b^{\lceil \alpha n \rceil}},\frac{1}{b^{\lceil \alpha n \rceil}}\right) \cap C(b,D) = C(b,D).$$
Therefore \cref{thm:large_intersection} implies that $$W_{t}(\psi) \in \mathcal{G}^{s}(C(b,D)).$$
\end{proof}

{\noindent \bf  Acknowledgements}. The authors thank Bo Wang for helpful discussions, in particular an observation that leads to Corollary \ref{cor:dim_same_prime}. We also thank the anonymous referee for several suggestions that improve the expression of this paper. B. Li was supported by NSFC 12271176 and Guangdong Natural Science Foundation 2024A1515010946. R. Li (corresponding author) was supported by NSFC 12401006 and Guangdong Basic and Applied Basic Research Foundation 2023A1515110272. Y. F. Wu was supported by NSFC 12301110 and Natural Science Foundation of Hunan Province 2023JJ40700.

\bibliographystyle{abbrv}

\end{document}